\documentclass[a4paper]{article}

\usepackage[dvips]{graphics}
\usepackage[dvips]{graphicx}
\usepackage{amsfonts}
\usepackage{amssymb}
\usepackage{amsmath}
\usepackage{amstext}
\usepackage{amsbsy}
\usepackage{amsopn}
\usepackage{amsthm}
\usepackage{color}
\usepackage{amscd}
\usepackage{enumerate}
\usepackage{amsxtra}
\usepackage{mathrsfs}
\usepackage{upref}
\usepackage[colorlinks,
linkcolor=red,
anchorcolor=red,
citecolor=red
]{hyperref}
\usepackage{lineno}
\usepackage{geometry}
\newcommand{\A}{\mathcal{A}}
\newcommand{\EE}{\mathbb{E}}
\newcommand{\RR}{\mathbb{R}}
\newcommand{\PP}{\mathbb{P}}
\newcommand{\B}{\mathcal{B}}
\newcommand{\mC}{\mathcal{C}}

\newcommand{\wh}[1]{\widehat{#1}}

\newcommand{\nel}{\left\|}
\newcommand{\ner}{\right\|}

\newcommand{\norm}[1]{\lVert#1\rVert}

\newtheorem{thm}{Theorem}[section]

\newtheorem{proposition}{Proposition}[section]
\newtheorem{lem}{Lemma}[section]

\theoremstyle{definition}
\newtheorem{defn}{Definition}[section]
\theoremstyle{remark}

\begin{document}
\title{Learning interacting particle systems: \\diffusion parameter
	estimation for aggregation equations}

\author{Hui Huang\thanks{Department of Mathematics, Simon Fraser University, Burnaby, BC, Canada. Email: hha101@sfu.ca}, 
	Jian-Guo Liu\thanks{Department of Mathematics and Department of Physics,
		Duke University, Durham, NC, USA. Email: jliu@phy.duke.edu}, 
	Jianfeng Lu\thanks{Department of Mathematics, Department of Physics and
		Department of Chemistry, Duke University, Durham, NC,
USA. Email: jianfeng@math.duke.edu}}
\maketitle
\begin{abstract}
	In this article, we study the parameter estimation of interacting particle systems subject to the Newtonian aggregation and Brownian diffusion. Specifically, we construct an estimator $\widehat{\nu}$ with partial observed data to approximate the diffusion parameter $\nu$, and the estimation error is achieved. Furthermore, we extend this result to general aggregation equations with a bounded Lipschitz interaction field.
\end{abstract}	
{\small {\bf Keywords:}
	Inverse problem, parameter identification of agent based model, mean-field limit, data assimilation, concentration inequality, discrete observation.}
\section{Introduction}
Parameter estimation of (stochastic) dynamical systems is an
  exciting area of research with ubiquitous applications in many areas
  in science and technology, where it usually requires  incorporating data into a model. This is often known as data assimilation
  (see the recent book \cite{LawStuartZygalakis:15} for a mathematical
  introduction) in particular in the context of numerical weather
  forecast. It is also known as system identification in the control
  literature (see for example \cite[Chapters 11 and
  12]{khalil2004modeling} for applications for modeling robots). In
  such problems, a physical model of the form of a dynamical system is
  derived from (partial) empirical observations and is usually
  calibrated with and improved by experimental data. The problem is
  also related to uncertainty quantification, which is important as
  it enables building of more realistic models and making better
  predictions of their behavior in the future. 
   In the modeling of
  self-organized systems, different ways to qualify uncertainties have
  been studied (see for example \cite{albi2015uncertainty, Fornasier,
    dimarco2017uncertainty, during2009boltzmann, katz2011inferring,
    tosin2017boltzmann}). 

In this work, we are interested in the parameter estimation
  problems arising from a particular class of physical systems that
  can be modeled by interacting particle systems. This means that the
  dynamics of the system is determined by interactions between agents
  (particles) together with some intrinsic or extrinsic random
  effects. Such systems are widely used to establish different
  mathematical models describing collective behaviors of organisms and
  social aggregations, for instance flocks of birds
  \cite{hildenbrandt2010self}, aggregation of bacteria
  \cite{bellomo2015toward}, schools of fish \cite{hemelrijk2008self},
  swarms formed by insects \cite{bernoff2013nonlocal}, opinion
  dynamics \cite{motsch2014heterophilious} and robotics and space
  missions \cite{ji2007distributed}.  Various types of diffusion are
  considered in these models: While linear diffusion is more commonly
  used \cite{dolbeault2004optimal}, the diffusion can be slow in areas
  with few particles, known as the degenerate (slow) diffusion model
  \cite{topaz2006nonlocal}; and similarly, the diffusion can also be
  fast \cite{sugiyama2011extinction}. One may also consider the
  nonlocal diffusion, where organisms adopt L\'{e}vy process search
  strategies which have continuous paths interspersed with random
  jumps \cite{huang2016well}. Thus qualifying the type of the
  diffusion can significantly reduce the uncertainty in model
  predictions and is hence a very important step in many
  applications. Our present paper focuses on the case of Brownian
  diffusion with unknown diffusion parameter. We study the diffusion
  parameter estimation of such interacting particle systems with
  partial observed data.  

More precisely, the microscopic agent-based model investigated here describes the evolution of positions
of $N$ agents, denoted by $\{X_i^t\} \subset \RR^d$,
$i = 1, \ldots, N$, whose evolution is governed by a system of
stochastic differential equations (SDEs) of the type
\begin{equation}\label{eq:agent}
d X_i^t= \frac{1}{N-1} \sum_{j\neq i}^N F(X_i^t- X_j^t) d t + \sqrt{2 \nu} d B_i^t,\quad i=1,\cdots,N,
\end{equation}
where $F$ models some pairwise interaction between the agents and
$B_i^t$ are independent realizations of Brownian motions which count
for extrinsic random perturbation of the agent positions. In such
systems, the agents are assumed to be identical, so that the noise
level $\nu$ is the same for each agent. In this work, we assume that
the interaction kernel $F$ between agents is known, while the noise
level $\nu$ is to be determined. More specifically, we will focus on
the case when the interaction between agents is given by Newtonian
type interaction for dimension $d\geq2$, or more precisely, a
regularized Newtonian interaction, to be specified below.  Suppose we
observe or track the trajectories of $K$ agents on the time interval
$[0, T]$, where $1\ll K \ll N$, the question we address in this work
is \emph{how to estimate $\nu$ and to quantify the error of the
	estimator}.

A more general situation one may consider is the problem for which the
interaction kernel is also to be determined, this will be left for
future works. In \cite{fetecau2011swarm}, authors solved the following inverse
problem for aggregation equations: given a  equilibrium
state, they constructed a corresponding force $F$ to ensure that equilibrium.
We also note the recent work \cite{Fornasier} which considers
learning the interaction kernel for a deterministic interacting
particle systems through a variational approach.  While
	admittedly that we have taken a simple scenario and a somewhat
	simplistic model for interacting agents, already many interesting
	issues arise from both mathematical and application point of
	view. For instance, how accurate one can make the estimation by only
	observing / tracking a few agents. How the potential
	singularities of the interacting potential (such as Coulomb or
	Newtonian type) impact the estimation accuracy.

Observe that the scaling of \eqref{eq:agent} is chosen such that we
are in the mean-field regime, as the interaction strength decreases as
$1/N$ as the number of agents $N \to \infty$. It is thus expected that
in the limit $N \to \infty$, the system can be well described by a
mean-field dynamics, which can be described as the following 
nonlinear partial differential equation (PDE)
\begin{subequations}\label{KS}
	\begin{align}
	&\partial_t\rho=\nu\Delta\rho-\nabla\cdot(\rho F\ast \rho),\quad x\in\mathbb{R}^d,~t>0,\\
	&\rho(x,0)=\rho_0(x), 
	\end{align}
\end{subequations}
where the noise level $\nu > 0$ enters the PDE system as a diffusion
parameter. In particular, here the interaction kernel is chosen as Newtonian: 
\begin{equation}\label{F}
F(x)=\mp\frac{C_*x}{|x|^d},\quad \forall~ x\in\mathbb{R}^d\backslash\{0\}, ~d\geq2,
\end{equation} with $C_*=\frac{\Gamma(d/2)}{2\pi^{d/2}}$. Here the sign $\mp$ indicates that the interaction between individuals can either be attraction or repulsion. Specifically, when the mechanism of interaction is attraction, the mean field equation \eqref{KS} becomes the parabolic-elliptic Keller-Segel  equation \cite{keller1970initiation,PCS}, which
is a prototypical model for chemotaxis and has been used in many
related modeling scenarios. The analysis of 
the scaling limit of interacting particle system \eqref{eq:agent} is usually called the \textit{mean-field limit}, which pass  limits from microscopic discrete particle systems to macroscopic continuum models.

While it would be intriguing to study the parameter identification
problem for the particle system \eqref{eq:agent} with the Newtonian
interaction \eqref{F}, such microscopic system is however ill-posed,
as shown by the recent deep result by Fournier and Jourdain
\cite[Proposition 4]{Fournier}: For any $N\geq 2$ and $T>0$, denote
$\{X_i(t); t\in [0,T], i = 1, \ldots, N\}$ the solution to
\eqref{eq:agent} with $F$ given in \eqref{F}, then
\begin{equation*}
\PP\left(\exists s\in[0,T], \exists 1\leq i<j\leq N: X_i(s)=X_j(s)\right)>0,
\end{equation*}
i.e., the singularity cannot be avoided in any finite time with a
positive probability and thus the particle system is not
well-defined. 

Classical results of the mean-field limit requires the kernel $F\in W^{1,\infty}(\RR^d)$. One possible way to overcome singularity is to regularize the kernel $F$. In particular, in this work  we consider the regularized kernel $F^N$:
\begin{equation}\label{kernel}
F^N=F\ast\psi_N,\quad
\psi_N(x)=N^{d\delta}\psi(N^\delta x), 
\end{equation}
where $\delta$ the cut-off index and $0 \leq \psi(x) \in C_0^{\infty}(\mathbb{R}^d)$ is a cut-off
function, which satisfies $\psi(x)=\psi(|x|)$ and
$\int_{\mathbb{R}^d}\psi(x)\,dx=1$.  
Then we have the regularized stochastic particle system $\{X_i^t\}_{i=1}^{N}$ satisfying 
\begin{equation}\label{particl system}
dX_{i}^t=\frac{1}{N-1}\sum_{j\neq i}^{N}F^N\big(X_{i}^t-X_{j}^t\big)\,dt+\sqrt{2\nu}\,dB_i^t,\quad i=1,\cdots,N,
\end{equation}
where the initial data $\{X_i^0\}_{i=1}^N$ are independently  identically distributed (i.i.d.) with the common density function $\rho_0$. Since the regularized kernel is Lipschitz for any fixed $N$, the system above has a unique global strong solution. The corresponding aggregation equation has the form
\begin{subequations}\label{RKS}
	\begin{align}
	&\partial_t\rho=\nu\Delta\rho-\nabla\cdot(\rho F^N\ast \rho),\quad x\in\mathbb{R}^d,~t>0,\\
	&\rho(x,0)=\rho_0(x).
	\end{align}
\end{subequations}

 Classical results for mean-field limit with globally Lipschitz forces was obtained by Braun and Hepp \cite{braun1977vlasov} and Dobrushin \cite{dobrushin1979vlasov}.  Then Bolley, Ca\~{n}izo and Carrilo \cite{bolley2011stochastic} presented an extension of the classical theory to the particle system with only locally Lipschitz interacting force. 
The last few years have seen great progress in  mean-field limits for singular forces by treating them with an $N$-dependent cut-off. In particular, 
the mean-field limit for the Keller-Segel model has been rigorously proved in \cite{fetecau2018propagation,Fournier,garcia2017,HH2,HH1}. 
And the deterministic particle method for aggregation equations can be found in \cite{carrillo2017blob,craig2016blob}.
For a general overview of this topic we refer readers to \cite{carrillo2010particle,golse2016dynamics,jabin2014review,jabin2017mean,spohn2004dynamics}.

Considering the parameter estimation problem for diffusion processes,
there is a huge literature in statistics and econometrics, often
related to the estimation of volatility in financial models. A
complete literature review is beyond our scope and we refer the
readers to the book \cite{Rao:2010} for an overview. To make the scenario more
realistic, instead of assuming the availability of some trajectories
$\{X_i^t\}_{i=1}^{N}$ for all time $t \in [0, T]$, we consider the
case that trajectories are only observed at discrete time snapshots
during the time interval. Diffusion parameter estimation problems based on discrete observations have been discussed by many authors \cite{ait2004estimators,bandi2017functional,bibby1995martingale,dacunha1986estimation,dohnal1987estimating,fan,huzak2016approximate,kessler1997estimation,yoshida1992estimation}. However, to our knowledge, no previous work has been done for diffusion estimation in the
context of interacting particle systems. Specifically, there are a few differences between
our work with these works: 1) We consider parameter estimation of an
interacting particle system, however authors mentioned above studied a
single diffusion process. 2) Our estimator
\eqref{estimator1} concerns the information of interacting particles, but they only
investigated one trajectory and take the expectation value of this stochastic process.  3) In our setting, the interacting force $F$ is singular, while
the drift function is assumed to be regular enough in usual
statistics literature as mentioned earlier.
Our main result, given below in
Theorem~\ref{mainthm} after we make precise the estimator $\widehat{\nu}$,
quantifies the estimation error of the proposed estimator.

Take a time step $\Delta t>0$ and let $t_n:=n\Delta t$
and $M:=\frac{T}{\Delta t}$ (we assume that $\frac{T}{\Delta t}$ is an
integer). Denote $X_i^{(n)}:=X_i^{t_n}$ as the solution to \eqref{particl system} at time $t_n$.  Namely, one has
\begin{align}
X_i^{(n+1)}-X_i^{(n)}&=\int_{t_n}^{t_{n+1}}\frac{1}{N-1}\sum_{j\neq i}^{N}F^N\big(X_{i}^s-X_{j}^s\big)\,ds+\sqrt{2\nu }(B_i^{t_{n+1}}-B_i^{t_n})\notag \\
&= \int_{t_n}^{t_{n+1}}\frac{1}{N-1}\sum_{j\neq i}^{N}F^N\big(X_{i}^s-X_{j}^s\big)\,ds+\sqrt{2\nu \Delta t}\mathcal{N}_i^{(n)},
\end{align}
where $\mathcal{N}_i^{(n)}\sim \mathcal{N}(0,1)^d$, i.e.  the standard Gaussian distribution in dimension $d$.

Then 
we are ready to define our estimator for the diffusion parameter
\begin{equation}\label{estimator1}
\widehat\nu:=\frac{1}{2dKT} \sum_{i=1}^K\sum_{n=0}^{M-1} \left | X_{i}^{(n+1)}-X_{i}^{(n)}\right|^2,
\end{equation}
where $1\ll K \ll N$, which means we only have partial observations.

Our main result quantifies the estimation error of the proposed estimator \eqref{estimator1}, which is summarized as below 
\begin{thm}\label{mainthm}
	Suppose the initial data $0\leq\rho_0(x)\in L^1\cap L^\infty(\mathbb{R}^d)$ and let
	$\rho(x,t)$ be the regular solution of the aggregation equation \eqref{KS}
	up to the time $T$ such that $\rho\in L^\infty(0,T;L^1\cap L^\infty(\mathbb{R}^d))$.  Take a time step $\Delta t>0$ and
	let $t_n:=n\Delta t$ and $M:=\frac{T}{\Delta t}$. Assume
	$\{X_i^{(n)}\}_{i=1,n=0}^{K,M}$ be the $K$ $(1\ll K\ll N)$ sample
	trajectories satisfying \eqref{particl system} with the cut-off
	index $0<\delta<\frac{1}{d}$ at time $t_n$. For any $\alpha>0$,
	there exists some constant $N_0>0$ depending only on $\nu$, $\alpha$, $T$ and
	$\|\rho_0\|_{L^1\cap L^\infty(\mathbb{R}^d)}$, such that for
	$N\geq N_0$, the estimator $\widehat{\nu}$ defined in \eqref{estimator1}
	is an approximation of $\nu$, and the following estimate holds
	\begin{equation}\label{maineq}
	\PP\Bigl(|\wh{\nu}-\nu|\leq C_\alpha\nu^{\frac{1}{2}} \Delta t^{\frac{1}{2}} (1+ \nu^{\frac{1}{2}}N^{-\delta}\log(N))+\gamma \nu\Bigr)\geq 1-N^{-\alpha}-2e^{-\frac{dKM\gamma^2}{8}},
	\end{equation}
	for any $\gamma\in(0,1)$, where $C_\alpha>0$ depends only
	on $\alpha$, $T$ and
	$\|\rho_0\|_{L^1\cap L^\infty(\mathbb{R}^d)}$.
\end{thm}

Let us remark on two simple consequences from \eqref{maineq}. 
	If we consider $N \to \infty$, \eqref{maineq} simplifies to 
	\begin{equation}\label{maineqNinf}
	\PP\Bigl(|\wh{\nu}-\nu|\leq C_\alpha\nu^{\frac{1}{2}} \Delta t^{\frac{1}{2}} + \gamma \nu\Bigr)\geq 1-2e^{-\frac{dKM\gamma^2}{8}}
	\end{equation}
	for any $\gamma \in (0, 1)$. Thus despite that we are dealing with
	an interacting system, when $N$ is large the accuracy of the
	estimator is similar to that based on using $K$ (independent)
	trajectory observations of a \emph{non-interacting} particle system.
	Moreover to get from \eqref{maineq} a simpler looking error bound,
	we can choose $\Delta t^{\frac{1}{2}}=\gamma$ (assuming we
	are able to adjust the frequency of observations) and get
	\begin{align}\label{maineq1}
	&\PP\left(|\widehat{\nu}-\nu|\leq (C_\alpha\nu^{\frac{1}{2}}+\nu)\Delta t^{\frac{1}{2}}\right)
	\geq1-2e^{-\frac{dKT}{8}},
	\end{align}
	where we used the fact that $M\gamma^2=\frac{T}{\Delta t}\gamma^2=T$. Estimate \eqref{maineq1} indicates that increasing the number $K$ of the observed data improves the accuracy of our estimator $\widehat \nu$ (the probability increase as $K$ becomes larger).

To prove the theorem on the error of the estimator, we defined a intermediate estimator
\begin{equation}
 \nu_{K,N}:=\frac{1}{2dKT} \sum_{i=1}^K\sum_{n=0}^{M-1} \left \lvert X_{i}^{(n+1)}-X_{i}^{(n)}-\int_{t_n}^{t_{n+1}}\frac{1}{N-1}\sum_{j\neq i}^{N}F^N\big(X_{i}^s-X_{j}^s\big)ds\right \rvert^2,
\end{equation}
then we split the error into two parts:
\begin{align}\label{split1}
|\widehat{\nu}-\nu|\leq |\widehat{\nu}-\nu_{K,N}|+|\nu_{K,N}-\nu|.
\end{align}
Let us denote 
\begin{align}
a_i^n:=X_{i}^{(n+1)}-X_{i}^{(n)},\quad b_i^n:= \int_{t_n}^{t_{n+1}}\frac{1}{N-1}\sum_{j\neq i}^{N}F^N\big(X_{i}^s-X_{j}^s\big)ds,
\end{align}
then
\begin{align}
|\nu_{K,N}-\widehat{\nu}|&=\frac{1}{2dKT} \sum_{i=1}^K\sum_{n=0}^{M-1} \left||a_i^n-b_i^n|^2- |a_i^n|^2\right|\notag\\
&\leq \frac{1}{2dKT} \sum_{i=1}^K\sum_{n=0}^{M-1}  |b_i^n|^2 +\frac{1}{dKT} \sum_{i=1}^K\sum_{n=0}^{M-1} |a_i^nb_i^n| \notag\\
&\leq \frac{1}{2dKT} \sum_{i=1}^K\sum_{n=0}^{M-1}  |b_i^n|^2 +(\frac{1}{dKT}\sum_{i=1}^K\sum_{n=0}^{M-1} |a_i^n|^2)^\frac{1}{2}(\frac{1}{dKT}\sum_{i=1}^K\sum_{n=0}^{M-1} |b_i^n|^2)^\frac{1}{2}\notag\\
&\leq\frac{1}{2dKT} \sum_{i=1}^K\sum_{n=0}^{M-1}  |b_i^n|^2  +(2\widehat \nu)^{\frac{1}{2}}(\frac{1}{dKT}\sum_{i=1}^K\sum_{n=0}^{M-1} |b_i^n|^2)^\frac{1}{2}.
\end{align}
Notice that
\begin{align}
&\frac{1}{2dKT} \sum_{i=1}^K\sum_{n=0}^{M-1}  |b_i^n|^2=\frac{1}{2dKT} \sum_{i=1}^K\sum_{n=0}^{M-1} \left| \int_{t_n}^{t_{n+1}}\frac{1}{N-1}\sum_{j\neq i}^{N}F^N\big(X_{i}^s-X_{j}^s\big)ds\right|^2\notag \\
\leq & \frac{1}{dKT}\sum_{i=1}^K\sum_{n=0}^{M-1} \left| \int_{t_n}^{t_{n+1}}\left(\frac{1}{N-1}\sum_{j\neq i}^{N}F^N\big(X_{i}^s-X_{j}^s\big) -\int_{\mathbb{R}^d}F^N(X_i^{s}-y)\rho(y,s)dy\right)ds\right|^2 \notag \\
&+\frac{1}{dKT} \sum_{i=1}^K\sum_{n=0}^{M-1}\left|\int_{t_n}^{t_{n+1}}\int_{\mathbb{R}^d}F^N(X_i^{s}-y)\rho(y,s)dyds\right|^2\notag\\
:=&|\mathcal{I}_2|+|\mathcal{I}_3|,
\end{align}
which implies
\begin{equation}
|\nu_{K,N}-\widehat{\nu}|\leq |\mathcal{I}_2|+|\mathcal{I}_3|+(2\widehat \nu)^{\frac{1}{2}}(2|\mathcal{I}_2|+2|\mathcal{I}_3|)^{\frac{1}{2}}.
\end{equation}

Collecting above inequality and \eqref{split1}, we concludes
\begin{align}\label{splitm}
|\widehat{\nu}-\nu|\leq  (1+2\widehat{\nu}^\frac{1}{2})(|\mathcal{I}_2|^{\frac{1}{2}}+|\mathcal{I}_3|^{\frac{1}{2}})+|\nu_{K,N}-\nu|.
\end{align}
when $|\mathcal{I}_2|, |\mathcal{I}_3|<1$.

Moreover, notice that
\begin{align}\label{boundedhatnu}
\widehat{\nu} =\frac{1}{2dKT} \sum_{i=1}^K\sum_{n=0}^{M-1} \left | a_i^n\right|^2&\leq  \frac{1}{dKT} \sum_{i=1}^K\sum_{n=0}^{M-1} \left | a_i^n-b_i^n\right|^2+ \frac{1}{dKT}\sum_{i=1}^K\sum_{n=0}^{M-1} \left | b_i^n\right|^2
 \notag\\
&\leq  2\nu_{K,N}+2|\mathcal{I}_2|+2|\mathcal{I}_3|\notag\\
&\leq 2 |\nu_{K,N}-\nu|+2|\mathcal{I}_2|+2|\mathcal{I}_3|+2\nu\leq 6+2\nu,
\end{align} 
when $|\nu_{K,N}-\nu|,|\mathcal{I}_2|, |\mathcal{I}_3|<1$.
Combing \eqref{boundedhatnu} with \eqref{splitm}, it yields that
\begin{equation}\label{split}
|\widehat{\nu}-\nu|\leq C \nu^{\frac{1}{2}}(|\mathcal{I}_2|^{\frac{1}{2}}+|\mathcal{I}_3|^{\frac{1}{2}})+|\nu_{K,N}-\nu|,
\end{equation}
where $C$ is a positive number.

In the sequel, we will see that the estimate of $|\mathcal{I}_3|$ is a
direct result from the property of regularized kernel $F^N$ (see estimate \eqref{I3}). And the
estimate of $|\mathcal{I}_2|$  is an estimate of interaction (see Theorem \ref{thminteract}),
which follows from the mean-field limit result (see Theorem \ref{thmmean}). As for the estimate of
$|\nu_{K,N}-\nu|$, it can be deduced from a concentration inequality of
Chi-squared distribution (see Theorem \ref{thm3}).

The work is organized as follows: In the next section, we will give a rigorous proof of the mean-field limit for aggregation equations with Newtonian potential. Base on this, we also obtain an error estimate on interaction. Section 3 is devoted to prove that our estimator $\wh{\nu}$ is a good approximation of $\nu$ and the convergence rate between them is achieved. Then in Section 4 we further extend our result to the case where the aggregation equation has a bounded Lipschitz interacting force.

\section{Mean-field limit and estimate on interaction}

In this section we will prove the mean-field limit for particle system \eqref{particl system}. Namely, given the solution $\rho$ to the mean-field  equation \eqref{RKS}, we construct a mean-field trajectories $\{Y_i^t\}_{i=1}^N$ from  \eqref{RKS}, then we  prove the closeness between $X_t=(X_1^t,\cdots,X_N^t)$ and $Y_t=(Y_1^t,\cdots,Y_N^t)$. 
To do this, we shall consider again a Newtonian system with noise. However, this time not subject to the pair interaction but under the influence of the external mean field $F^N\ast \rho(x,t)$
\begin{equation}\label{meandynamics}
dY_{i}^t=\int_{\RR^d}F^N\big(Y_i^t-y\big)\rho(y,t)dy\,dt+\sqrt{2\nu}\,dB_i^t,\quad i=1,\cdots,N,
\end{equation}
here we let $\{Y_i^t\}_{i=1}^N$ has the same initial condition as
$\{X_i^t\}_{i=1}^N$ (i.i.d. with the common density $\rho_0$).  Since the
particles are subject to an external field, the independence is
conserved. Therefore the $\{Y_i^t\}_{i=1}^N$ are distributed i.i.d.
according to the common probability density $\rho_t$. We remark that the
aggregation equation \eqref{RKS} is Kolmogorov's forward equation for any
solution of \eqref{meandynamics}, and in particular their probability
distribution $\rho_t$ solves \eqref{RKS}.

\subsection{Preliminaries}
\textbf{Notations:}
The generic constant will be denoted generically by $C$, even if it is different from line to line. The notation $\norm{\cdot}_{p}$ represents the usual $L^p$-norm of a function for any $1\leq p\leq \infty$. For a vector $X_t=(X_1^t,\dots,X_N^t)$, we denote
\begin{equation}
\norm{X_t}_\infty:=\sup\limits_{i=1,\cdots,N}|X_i^t|.
\end{equation}

Since error estimates obtained later are valid when the solution of PDE \eqref{RKS} is regular enough, we assume that
\begin{equation}\label{initial}
0\leq\rho_0\in L^1\cap L^\infty(\mathbb{R}^d),
\end{equation}
then equation \eqref{RKS} has a unique local solution with the following regularity
\begin{equation}
\|\rho\|_{L^\infty\left(0,T; L^1\cap L^\infty(\mathbb{R}^d)\right)}\leq C\left(\|\rho_0\|_{L^1\cap L^\infty(\mathbb{R}^d)}\right)=:C_{\rho_0},
\end{equation}
where $C_{\rho_0}$ is independent of $N$ and $T>0$ depends only on $\nu$ and $\|\rho_0\|_{L^1\cap L^\infty(\mathbb{R}^d)}$. The proof of this result is a standard process (see for example \cite[Proposition 4.1]{fetecau2018propagation}). 

Let us recall some estimates of the regularized kernel $F^N$ defined in \eqref{kernel}:
\begin{lem}\label{lmkenerl}(\cite[Lemma 2.1]{HH2})
	\begin{enumerate}[(i)]
		\item $F^N(0)=0$, $F^N(x)=F(x)$ for any $|x|\geq N^{-\delta}$ and  
		$|F^N(x)|\leq |F(x)|$ for  
			any $x\in\mathbb{R}^d$;
		\item $|\partial^\beta F^N(x)|\leq CN^{(d+|\beta|-1)\delta} \mbox{ for  
			any } x\in\mathbb{R}^d$;
		\item $\|F^N\|_{2}\leq CN^{(\frac{d}{2}-1)\delta}.$
	\end{enumerate}
\end{lem}

Next we define a cut-off function $L^N$, which will provide the local Lipschitz bound for $F^N$.
\begin{defn}\label{defLN}
	Let
	\begin{equation}\label{lN}
	L^N(x)=\left\{
	\begin{aligned}
	&\frac{6^d}{|x|^{d}},  && \text{ if } |x|\geq 6N^{-\delta},\\
	&N^{d\delta},  && \text{ else },
	\end{aligned}
	\right.
	\end{equation}
	and $\mathcal{L}^N: \RR^{dN}\rightarrow \RR^N$ be defined by  
	$(\mathcal{L}^N(X_t))_i:=\frac{1}{N-1}\sum\limits_{i\neq j}^NL^N(X_i^t-X_j^t)$.  
	Furthermore, we define $\overline {\mathcal{L}}^N(Y_t)$ by $( \overline  
	{\mathcal{L}}^N( Y_t))_i:=\int_{\RR^d} L^N(Y_i^t-x)\rho(x,t)dx$.
\end{defn}
Denote
\begin{equation}\label{FN}
(\mathcal{F}^N(Y^t))_i:=\frac{1}{N-1}\sum_{j\neq i}^NF^N(Y_i^t-Y_j^t),
\end{equation}
then we have the local Lipschitz continuity of $\mathcal{F}^N$:
\begin{lem}{\cite[Lemma 2.3]{vlasovhui}}\label{lmlip}
	If $\|X_t-Y_t\|_\infty\leq 2N^{-\delta}$, then it holds that
	\begin{equation}\label{lmlipeq}
	\| \mathcal{F}^N(X_t)-\mathcal{F}^N(Y_t)\|_\infty\leq C\|\mathcal{L}^N(Y_t)\|_\infty\|X_t-Y_t\|_\infty,
	\end{equation}
	for some $C>0$ independent of $N$.
\end{lem}

The following observations of $F^N$ and $L^N$ turn out to be very helpful in the sequel:
\begin{lem}{\cite[Lemma 2.4]{vlasovhui}}\label{converse}
	Let $L^N(x)$ be defined in Definition \ref{defLN} and $\rho \in L^{1}\cap L^{\infty} (\RR^d)$.  Then there exists a constant $C>0$ independent of $N$ such that
	\begin{equation}\label{con1}
	\norm{L^N\ast \rho}_\infty\leq C\log(N)(\norm{\rho}_1+\norm{\rho}_\infty),\quad\norm{(L^N)^2\ast \rho}_\infty\leq CN^{d\delta}(\norm{\rho}_1+\norm{\rho}_\infty);
	\end{equation}
	and
	\begin{equation}\label{con2}
	\norm{ F^N\ast \rho}_\infty\leq C(\norm{\rho}_1+\norm{\rho}_\infty),\quad \norm{\nabla F^N\ast \rho}_\infty\leq C\log(N)(\norm{\rho}_1+\norm{\rho}_\infty).
	\end{equation}
\end{lem}

Also, we need the following concentration inequality to  provide us the probability bounds of random variables:
\begin{lem}\label{central} 
	Let $Z_1,\cdots,Z_N$ be $i.i.d.$ random variables with $\mathbb{E}[Z_i]=0,$ $\mathbb{E}[Z_i^2]\leq g(N)$
	and $|Z_i|\leq C\sqrt{Ng(N)}$. Then for any $\alpha>0$, the sample mean $\bar{Z}=\frac{1}{N}\sum_{i=1}^{N}Z_i$ satisfies
	\begin{equation}
	\PP\left(|\bar{Z}|\geq\frac{C_\alpha \sqrt{g(N)}\log(N)}{\sqrt{N}}\right)\leq N^{-\alpha},
	\end{equation}
	where $C_\alpha$ depends only on $C$ and $\alpha$.
\end{lem}
The proof can be seen in \cite[Lemma 1]{GJ}, which is a direct result of the Taylor expansion and the Markov's inequality.

\subsection{Mean-field limit for the aggregation equation with Newtonian potential}
In this section, we obtain the maximal distance between the exact microscopic dynamics \eqref{particl system} and the approximate mean-field dynamics \eqref{meandynamics}.
Denote
\begin{equation}\label{barFN}
(\overline{ \mathcal{F}}^N(Y_t))_i:=\int_{\RR^d} F^N(Y_i^t-x)\rho(x,t)dx,
\end{equation}
then we can introduce the following lemma of law of large numbers:
\begin{lem}\label{lmlarge} At any fixed time $t\in[0,T]$, suppose that $Y_t=(Y_i^t)_{i=1,\cdots,N}$ satisfies the mean-field dynamics \eqref{meandynamics} with i.i.d initial data sharing the common density $\rho_0$ satisfying \eqref{initial}.  Assume that $\mathcal{F}^N$ and $\overline {\mathcal{F}}^N$ are defined in \eqref{FN} and \eqref{barFN} respectively,  $\mathcal{L}^N$ and $\overline {\mathcal{L}}^N$ are showed in Definition \ref{defLN}. For any $\alpha>0$ and $0<\delta\leq \frac{1}{d}$, there exist a constant $C_{1,\alpha}>0$ depending only on  $\alpha$, $T$ and $C_{\rho_0}$ such that
	\begin{equation}\label{largef}
	\PP\left(\nel \mathcal{F}^N(Y_t)-\overline{ \mathcal{F}}^N(Y_t)\ner_\infty\geq C_{1,\alpha} N^{\frac{\delta(d-2)-1}{2}}\log (N)\right)\leq N^{-\alpha},
	\end{equation}
	and
	\begin{equation}\label{largel}
	\PP\left(\nel \mathcal{L}^N(Y_{t})-\overline{\mathcal{L}}^N(Y_{t})\ner_\infty\geq C_{1,\alpha} N^{\frac{d\delta-1}{2}}\log (N)\right)\leq N^{-\alpha}.
	\end{equation}
\end{lem}
\begin{proof}
	We can prove this lemma by using Lemma \ref{central}.
	Due to the exchangeability of the particles, we are ready to bound
	\begin{equation}
	(\mathcal{F}^N(Y_t))_1-(\overline {\mathcal{F}}^N(Y_t))_1=\frac{1}{N-1}\sum_{j=2}^NF^N(Y_1^t-Y_j^t)-\int_{\RR^3} F^N(Y_1^t-x)\rho(x,t)dx=\frac{1}{N-1}\sum_{j=2}^{N}Z_j,
	\end{equation}
	where 
    \begin{equation*}
      Z_j:=F^N(Y_1^t-Y_j^t)-\int_{\RR^d} F^N(Y_1^t-x)\rho(x,t)dx.
    \end{equation*}
	Since $Y_1^t$ and $Y_j^t$ are independent when $j\neq 1$ and $F^N(0)=0$, let us consider $Y_1^t$ as given and denote $\mathbb{E'}[\cdot]=\mathbb{E}[\cdot|Y_1^t]$. It is easy to show that
	$\mathbb{E}'[Z_j]=0$ since
	\begin{align*}
	\mathbb{E}'\left[F^N(Y_1^t-Y_j^t)\right]=\int_{\RR^d} F^N(Y_1^t-x)\rho(x,t)dx.
	\end{align*}
	
	To use Lemma \ref{central}, we need a bound for the variance
	\begin{equation}
	\mathbb{E}'\big[|Z_j|^2\big]=\mathbb{E}'\left[\left|F^N(Y_1^t-Y_j^t)-\int_{\RR^3} F^N(Y_1^t-x)\rho(x,t)dx\right|^2\right].
	\end{equation}
	Since it follows from Lemma \ref{converse} that
	\begin{equation}
	\int_{\RR^3} F^N(Y_1^t-x)\rho(x,t)dx\leq C(\norm{\rho}_1+\norm{\rho}_\infty),
	\end{equation}
	it suffices to bound
	\begin{equation}
	\mathbb{E'}\big[F^N(Y_1^t-Y_j^t)\big]=\int_{\RR^d} F^N(Y_1^t-x)\rho(x,t)dx\leq C(\norm{\rho}_1+\norm{\rho}_\infty)\leq C(T,C_{\rho_0}),
	\end{equation}
	and
	\begin{equation}
	\mathbb{E'}\big[F^N(Y_1^t-Y_j^t)^2\big]=\int_{\RR^d} F^N(Y_1^t-x)^2\rho(x,t)dx\leq \norm{\rho}_\infty\norm{F^N}_2^2\leq C(T,C_{\rho_0})N^{\delta(d-2)},
	\end{equation}
	where we have used $\norm{F^N}_2\leq CN^{\delta(\frac{d}{2}-1)}$   in Lemma \ref{lmkenerl} $(iii)$.
	Hence one has
	\begin{equation}
	\mathbb{E}'\big[|Z_j|^2\big]\leq CN^{\delta(d-2)}.
	\end{equation}
	
	So the hypotheses of Lemma \ref{central} are satisfied with $g(N)=CN^{\delta(d-2)}$. In addition, it follows from $(ii)$ in Lemma \ref{lmkenerl} that $|Z_j|\leq CN^{\delta (d-1)}\leq C\sqrt{Ng(N)}$. Hence, using Lemma \ref{central}, we have the probability bound 
	\begin{equation}
	\PP\left(\left|(\mathcal{F}^N(Y_t))_1-(\overline {\mathcal{F}}^N(Y_t))_1\right|\geq C(\alpha,T,C_{\rho_0})N^{\frac{\delta(d-2)-1}{2}}\log (N)\right)\leq N^{-\alpha}.
	\end{equation}
	Similarly, the same bound must also apply hold to other term with $i=2,\cdots,N$, which leads to
	\begin{equation}\label{residual1'}
	\PP\left(\nel \mathcal{F}^N(Y_t)-\overline {\mathcal{F}}^N(Y_t)\ner_\infty\geq C(\alpha,T,C_{\rho_0}) N^{\frac{\delta(d-2)-1}{2}}\log (N)\right)\leq N^{1-\alpha}.
	\end{equation}
	Let $C_{1,\alpha}$ be the constant in  \eqref{residual1'}, we conclude \eqref{largef}.
	
	To prove \eqref{largel}, we follow the same procedure above
	\begin{equation}
	(\mathcal{L}^N(Y_t))_1-(\overline {\mathcal{L}}^N(Y_t))_1=\frac{1}{N-1}\sum_{j=2}^NL^N(Y_1^t-Y_j^t)-\int_{\RR^d} L^N(Y_1^t-x)\rho(x,t)dx=\frac{1}{N-1}\sum_{j=2}^{N}Z_j,
	\end{equation}
	where $$Z_j=L^N(Y_1^t-Y_j^t)-\int_{\RR^3} L^N(Y_1^t-x)\rho(x,t)dx.$$ 
	It is easy to show that $\mathbb{E}'[Z_j]=0$.
	To use Lemma \ref{central}, we need a bound for the variance. One computes that
	\begin{equation}
	\mathbb{E'}\big[L^N(Y_1^t-Y_j^t)\big]=\int_{\RR^d} L^N(Y_1^t-x)\rho(x,t)dx\leq C\log(N)(\norm{\rho}_1+\norm{\rho}_\infty)\leq C(T,C_{\rho_0})\log(N),
	\end{equation}
	and
	\begin{equation}
	\mathbb{E'}\big[L^N(Y_1^t-Y_j^t)^2\big]=\int_{\RR^d} L^N(Y_1^t-x)^2\rho(x,t)dx\leq CN^{d\delta}(\norm{\rho}_1+\norm{\rho}_\infty)\leq C(T,C_{\rho_0})N^{d\delta},
	\end{equation}
	where we have used the estimates of $L^N$ in Lemma \ref{converse}.
	Hence one has
	\begin{equation}
	\mathbb{E}'\big[|Z_j|^2\big]\leq CN^{d\delta}.
	\end{equation}
	
	So the hypotheses of Lemma \ref{central} are satisfied with $g(N)=CN^{d\delta}$. In addition, it follows from Definition \ref{defLN}  that $|Z_j|\leq CN^{d\delta}\leq C\sqrt{Ng(N)}$. Hence, we have the probability bound 
	\begin{equation}
	\PP\left(\left|(\mathcal{L}^N(Y_t))_1-(\overline {\mathcal{L}}^N(Y_t))_1\right|\geq C(\alpha,T,C_{\rho_0}) N^{\frac{d\delta-1}{2}}\log (N)\right)\leq N^{-\alpha},
	\end{equation}
	by Lemma \ref{central}, which leads to
	\begin{equation}\label{residual1''}
	\PP\left(\nel \mathcal{L}^N(Y_t)-\overline{\mathcal{L}}^N(Y_t)\ner_\infty\geq C(\alpha,T,C_{\rho_0}) N^{\frac{d\delta-1}{2}}\log(N)\right)\leq N^{1-\alpha}.
	\end{equation}
	Thus, \eqref{largel} follows from \eqref{residual1''}.
	
\end{proof}

Next we improve the consistency error to all time. To do this, we need the following lemma, where we temporarily set the time step size $\Delta t=t_{n+1}-t_n=N^{-\frac{\beta}{d}}$ with $\beta>2$, which is only for the purpose of proving Proposition \ref{propconsis} and Proposition \ref{propstab}.  Here $N^{-\frac{\beta}{d}}$ won't influence the choice of the $\Delta t$ in Theorem \ref{mainthm}.
\begin{lem}\label{lmB}
	Assume that the time step size $\Delta t=t_{n+1}-t_n=N^{-\frac{\beta}{d}}$ for $\beta>2$ and $Y_t$ satisfies the mean-field dynamics \eqref{meandynamics}. There exists some constant $C_B>0$ depending only on $T$ and $C_{\rho_0}$, such that it holds
	\begin{equation}
	\PP\left(\sup\limits_n\sup\limits_{t\in[t_n,t_{n+1}]}\nel Y_t-Y_{t_n}\ner_\infty\geq C_B\nu^{\frac{1}{2}}N^{-\frac{\beta+2}{2d}}\right)\leq C_BN^{\frac{2+\beta}{2d}}\exp(-C_BN^{\frac{\beta-2}{d}}).
	\end{equation}
\end{lem}
\begin{proof}
	Notice that for $t\in[t_n,t_{n+1}]$
	\begin{align}\label{Qdiff}
Y_t-Y_{t_n}&=\int_{t_n}^{t}\int_{\RR^d}F^N\big(Y_i^s-y\big)\rho(y,s)dy\,ds+\sqrt{2\nu\Delta t}(B^t-B^{t_n})\notag \\
	&=:I_1(t)+I_2(t).
	\end{align}
	It follows from Lemma \ref{converse} that
	\begin{equation}\label{I_1}
	\sup\limits_{t_n\leq t\leq t_{n+1}}\|I_1(t)\|_\infty\leq C\Delta t\leq CN^{-\frac{\beta}{d}},
	\end{equation}
	where $C$ depending only on $T$ and $C_{\rho_0}$.
	To estimate $I_2(t)$, recall a basic property of the Brownian motion \cite[Lemma 2.7]{HH1}:
	\begin{equation}\label{Bproperty}
	\PP\left(\sup\limits_{t\leq s\leq t+\Delta t}\|B^s-B^t\|_\infty\geq b\right)\leq C_1(\sqrt{\Delta t}/b)\exp(-C_2b^2/\Delta t),
	\end{equation}
	where $C_1$ and $C_2$ depend only on $d$. Choosing $b=N^{-\frac{1}{d}}$ in \eqref{Bproperty}, it leads to 
	\begin{equation}\label{I_2}
	\PP\left(\sup\limits_{t_n\leq t\leq t_{n+1}}\|I_2(t)\|_\infty \geq \sqrt{2\nu}N^{-\frac{\beta+2}{2d}}\right)\leq C_1N^{\frac{2-\beta}{2d}}\exp(-C_2N^{\frac{\beta-2}{d}}).
	\end{equation}
	
	Collecting \eqref{I_1} and \eqref{I_2}, it yields that
	\begin{equation}
	\PP\left(\sup\limits_n\sup\limits_{t\in[t_n,t_{n+1}]}\nel Y_t-Y_{t_n}\ner_\infty\geq C\nu^{\frac{1}{2}}N^{-\frac{\beta+2}{2d}}\right)\leq  C_1N^{\frac{2+\beta}{2d}}\exp(-C_2N^{\frac{\beta-2}{d}}),
	\end{equation}
	for $\beta>2$, which concludes the proof.
\end{proof}

Now we can prove the consistency error in all time. 
\begin{proposition}\label{propconsis}(Consistency) Let $Y_t=(Y_i^t)_{i=1,\cdots,N}$ satisfies the mean-field dynamics \eqref{meandynamics} with i.i.d initial data sharing the common density $\rho_0$ satisfying \eqref{initial}.  Assume that $\mathcal F^N$ and $\overline {\mathcal{F}}^N$ be defined in \eqref{FN} and \eqref{barFN} respectively. For any $\alpha>0$ and $0<\delta\leq \frac{1}{d}$, there exist a constant $C_{2,\alpha}>0$ depending only on depends on $\alpha$, $T$ and $C_{\rho_0}$ such that
	\begin{equation}\label{consistency}
	\PP\left(\sup\limits_{t\in[0,T]}\nel \mathcal{F}^N(Y_t)-\overline {\mathcal{F}}^N(Y_t)\ner_\infty\geq C_{2,\alpha} \nu^{\frac{1}{2}} N^{\frac{\delta(d-2)-1}{2}}\log (N)\right)\leq  N^{-\alpha},
	\end{equation}
	and
	\begin{equation}\label{consistency1}
	\PP\left(\sup\limits_{t\in[0,T]}\nel \mathcal{L}^N(Y_t)-\overline {\mathcal{L}}^N(Y_t)\ner_\infty\geq C_{2,\alpha}  \nu^{\frac{1}{2}}N^{\frac{d\delta-1}{2}}\log (N)\right)\leq  N^{-\alpha}.
	\end{equation}
\end{proposition}
\begin{proof}
Denote events:
\begin{equation}\label{eventH}
\mathcal{H}:=\left\{\sup\limits_n\sup\limits_{t\in[t_n,t_{n+1}]}\nel Y_t-Y_{t_n}\ner_\infty\leq C_B\nu^{\frac{1}{2}}N^{-\frac{\beta+2}{2d}}\right\},
\end{equation}
and
\begin{equation}
\mC_{t_n}:=\left\{\nel \mathcal{F}^N(Y_{t_n})-\overline {\mathcal{F}}^N(Y_{t_n})\ner_\infty\geq C_{1,\alpha} N^{\frac{\delta(d-2)-1}{2}}\log (N)\right\},
\end{equation}
where $C_B$ and $C_{1,\alpha} $ are used in Lemma \ref{lmlarge} and Lemma \ref{lmB} respectively.
According to the Lemma \ref{lmlarge} and Lemma \ref{lmB}, one has
\begin{equation}
\PP(\mC_{t_n}^c)\leq N^{-\alpha},\quad \PP(\mathcal{H}^c)\leq C_BN^{\frac{2+\beta}{2d}}\exp(-C_BN^{\frac{\beta-2}{d}}).
\end{equation}
for any $\alpha>0$ and $\beta>2$.

Furthermore,  we denote
\begin{equation}\label{Btn}
\B_{t_n}:=\left\{\nel \mathcal{L}^N(Y_{t_n})-\overline{\mathcal{L}}^N(Y_{t_n})\ner_\infty\leq C_{1,\alpha} N^{\frac{d\delta-1}{2}}\log (N)\right\},
\end{equation}
then under the event $\B_{t_n}$, it holds that
\begin{equation}\label{Btnresult}
\|\mathcal{L}^N(Y_{t_n})\|_\infty\leq \|\overline {\mathcal{L}}^N(Y_{t_n})\|_\infty+C_{1,\alpha} N^{\frac{d\delta-1}{2}}\log (N)\leq C(\alpha, T, C_{\rho_0})\log (N).
\end{equation}
and $\PP(\B_{t_n}^c)\leq N^{-\alpha}$ by Lemma \ref{lmlarge}.

For all $t\in[t_n,t_{n+1}]$, under the event $\B_{t_n}\cap\mC_{t_n}\cap \mathcal{H}$, we obtain
\begin{align*}
&\nel \mathcal{F}^N(Y_t)-\overline {\mathcal{F}}^N(Y_t)\ner_\infty\notag\\
\leq &\nel \mathcal{F}^N(Y_t)-{\mathcal{F}}^N(Y_{t_n})\ner_\infty+\nel \mathcal{F}^N(Y_{t_n})-\overline {\mathcal{F}}^N(Y_{t_n})\ner_\infty+\nel \overline {\mathcal{F}}^N(Y_{t_n})-\overline {\mathcal{F}}^N(Y_t)\ner_\infty\notag\\
\leq& C\|\mathcal{L}^N(Y_{t_n})\|_\infty\|Y_t-Y_{t_n}\|_\infty+C_{1,\alpha} N^{\frac{\delta(d-2)-1}{2}}\log (N)+C\log(N)\nel Y_{t}-Y_{t_n}\ner_\infty\notag\\
\leq&C(\alpha, T, C_{\rho_0}) \nu^{\frac{1}{2}}\log (N)N^{-\frac{\beta+2}{2d}}+C_{1,\alpha} N^{\frac{\delta(d-2)-1}{2}}\log (N)\notag\\
\leq&C(\alpha, T, C_{\rho_0})  \nu^{\frac{1}{2}}N^{\frac{\delta(d-2)-1}{2}}\log (N),\quad(\beta>(d-2)(1-d\delta))
\end{align*}
where in the second inequality we have used the local Lipschitz bound of $\mathcal{F}^N$
\begin{equation}
\nel \mathcal{F}^N(Y_t)-\mathcal{F}^N(Y_{t_n})\ner_\infty\leq C\|\mathcal{L}^N(Y_{t_n})\|_\infty\|Y_t-Y_{t_n}\|_\infty,
\end{equation}
under the event $\mathcal{H}$ (see in Lemma \ref{lmlip}).
It yields that
\begin{equation}
\sup\limits_{t\in[0,T]}\nel \mathcal{F}^N(Y_t)-\overline {\mathcal{F}}^N(Y_t)\ner_\infty\leq C(\alpha, T, C_{\rho_0})  \nu^{\frac{1}{2}}N^{\frac{\delta(d-2)-1}{2}}\log (N),
\end{equation}
holds under the event $\bigcap\limits_{n=0}^{M-1}(\B_{t_n}\cap\mC_{t_n})\cap\mathcal{H}$. Therefore
\begin{align}\label{68}
&\PP\left(\sup\limits_{t\in[0,T]}\nel \mathcal{F}^N(Y_t)-\overline {\mathcal{F}}^N(Y_t)\ner_\infty\geq C(\alpha, T, C_{\rho_0}) \nu^{\frac{1}{2}}N^{\frac{\delta(d-2)-1}{2}}\log (N)\right)\notag \\
\leq &\sum\limits_{n=0}^{M-1}P(\B_{t_n}^c)+\sum\limits_{n=0}^{M-1}P(\mC_{t_n}^c)+P(\mathcal{H}^c)\notag \\
\leq &TN^{-\frac{d\alpha-\beta}{d}}+TN^{-\frac{d\alpha-\beta}{d}}+C_BN^{\frac{2+\beta}{2d}}\exp(-C_BN^{\frac{\beta-2}{d}})
\leq N^{-\alpha'}.
\end{align}
Denote $C_{2,\alpha'}$ to be the constant $C(\alpha, T, C_{\rho_0})$ in \eqref{68}. Since $\alpha>0$ is arbitrary and so is $\alpha'$, hence \eqref{consistency} holds true. The proof of \eqref{consistency1} can be done similarly.
\end{proof}

In order to prove the convergence, we still need the stability result which states:
\begin{proposition}\label{propstab}
	(Stability)   Assume that trajectories $X_t=(X_i^t)_{i=1,\cdots,N}$, $Y_t=(Y_i^t)_{i=1,\cdots,N}$ satisfy \eqref{particl system} and \eqref{meandynamics} respectively with the initial data $X_0=Y_0$, which are i.i.d. sharing the common density $\rho_0$ satisfying \eqref{initial}. Let  events $\mathcal{B}_{t_n}$ and $\mathcal{H}$ be defined in \eqref{Btn} and  \eqref{eventH} respectively, $\mathcal F^N$ be defined in \eqref{FN} . Denote events:
	\begin{equation}\label{eventA}
	\A:=\left\{\sup\limits_{t\in[0,T]}\nel X_t-Y_t\ner_\infty< N^{-\delta}\right\},
	\end{equation}
	and
	\begin{equation*}
	\mathcal{S}(\Lambda):=\bigg\{\| \mathcal{F}^N(X_t)-{\mathcal{F}}^N(Y_t) \|_\infty \leq \Lambda\log(N)\nel X_t-Y_t\ner_\infty+\Lambda \nu^{\frac{1}{2}}\log(N)N^{-\frac{\beta+2}{2d}},~\forall~t\in[0,T]\bigg\}.
	\end{equation*}
For any $\alpha>0$, there exists some $C_{3,\alpha}>0$ depending only on $\alpha$, $T$ and $C_{\rho_0}$ such that
	\begin{align*}
  \bigcap\limits_{n=0}^{M-1}\B_{t_n}\cap \A\cap \mathcal{H}\subset \mathcal{S}(C_{3,\alpha}).
	\end{align*}
Here the event $\mathcal{S}(C_{3,\alpha})$ can be seen as the stability result and the events $\B_{t_n}$, $\A$ and $\mathcal{H}$ can be treated  as the stability conditions.

\end{proposition}
\begin{proof}
	
	First, we split $\mathcal{S}(\Lambda)$ into the union of non-overlapping sets $\{\mathcal{S}_n(\Lambda)\}_{n=0}^{M-1}(\Lambda)$, where
	\begin{align*}
	\mathcal{S}_n(\Lambda):=\bigg\{\| \mathcal{F}^N(X_t)-\mathcal{F}^N(Y_t) \|_\infty& \leq \Lambda\log(N)\nel X_t-Y_t\ner_\infty\notag \\
&\quad	+\Lambda\log(N)N^{-\frac{\beta+2}{2d}},~\forall~t\in[t_n,t_{n+1}]\bigg\}.
	\end{align*}
	Notice that for any $t\in[t_n,t_{n+1}]$, under the event $\A\cap\mathcal{H}$, one has
		\begin{align*}
	\sup\limits_{t\in[t_n,t_{n+1}]}\nel X_t-Y_{t_n}\ner_\infty&\leq \sup\limits_{t\in[t_n,t_{n+1}]}\nel X_t-Y_{t}\ner_\infty+\sup\limits_{t\in[t_n,t_{n+1}]}\nel Y_t-Y_{t_n}\ner_\infty\notag \\
	&\leq N^{-\delta}+C_B \nu^{\frac{1}{2}}N^{-\frac{\beta+2}{2d}}< 2N^{-\delta},\quad (\beta>2d\delta-2)
	\end{align*}
	and
	\begin{align*}
	\sup\limits_{t\in[t_n,t_{n+1}]}\nel Y_t-Y_{t_n}\ner_\infty<C_B \nu^{\frac{1}{2}}N^{-\frac{\beta+2}{2d}}<N^{-\delta}.\quad (\beta>2d\delta-2)
	\end{align*}
	Then applying the local Lipschitz bound of $F^N$ (see in Lemma \ref{lmlip}) leads to
	\begin{align*}
	\| \mathcal{F}^N(X_t)-\mathcal{F}^N(Y_t) \|_\infty &\leq \| \mathcal{F}^N(X_t)-\mathcal{F}^N(Y_{t_n}) \|_\infty +\| \mathcal{F}^N(Y_{t_n})-\mathcal{F}^N(Y_{t}) \|_\infty\notag\\
	&\leq  C\|\mathcal{L}^N(Y_{t_n})\|_\infty\left(\nel X_t-Y_{t_n}\ner_\infty+\nel Y_{t_n}-Y_{t}\ner_\infty\right)\notag\\
	&\leq C\|\mathcal{L}^N(Y_{t_n})\|_\infty\nel X_t-Y_t\ner_\infty+2C\|\mathcal{L}^N(Y_{t_n})\|_\infty\nel Y_{t_n}-Y_{t}\ner_\infty
	\end{align*}
under the event $\A\cap\mathcal{H}$.

Furthermore,  under the event $\B_{t_n}$, it follows from \eqref{Btnresult} that
	\begin{equation}
	\|\mathcal{L}^N(Y_{t_n})\|_\infty\leq C\log(N),
	\end{equation}	
	Hence, for all $t\in[t_n,t_{n+1}]$ one has
	\begin{align*}
	\| \mathcal{F}^N(X_t)-\mathcal{F}^N(Y_t) \|_\infty &\leq C(\alpha, T, C_{\rho_0})\log(N)\nel X_t-Y_t\ner_\infty\\
	&\quad+C(\alpha, T, C_{\rho_0}) \nu^{\frac{1}{2}}\log(N)N^{-\frac{\beta+2}{2d}},
	\end{align*}
	under event $\A\cap\mathcal{H}\cap\B_{t_n}$. Denote the $C(\alpha, T, C_{\rho_0})$ in the above as $C_{3,\alpha}$. This implies $\A\cap\mathcal{H}\cap\B_{t_n}\subset \mathcal{S}_n(C_{3,\alpha})$,
	 which yields $$\bigcap\limits_{n=0}^{M-1}\B_{t_n}\cap\mathcal{H}\cap \A\subset \mathcal{S}(C_{3,\alpha}).$$
	 	Thus, the proposition has been proved.

\end{proof}

 Before proving the result on mean-field limit,  let us recall a Gronwall-type inequality in \cite{vlasovhui}.
\begin{lem}\label{lmgron}
	For any $T>0$, let $e(t)$ be a non-negative continuous function on $[0,T]$ with the initial data $e(0)=0$ and $\lambda,~\delta$ be two universal constants satisfying $0<\delta<\lambda$. Assume that for any $0< T_1\leq T$ the function $e(t)$ satisfies the following differential inequality holds with $C>0$ independent of $N>0$
	\begin{equation}\label{priorineq}
	\frac{de(t)}{dt}\leq C\log(N)e(t)+C\log(N)N^{-\lambda},\quad 0<t\leq T_1,
	\end{equation}
	provided that 
	\begin{equation}\label{priorcon}
	\sup\limits_{t\in[0,T_1]}e(t)\leq N^{-\delta},
	\end{equation}
	holds.  Then $e(t)$ is uniformly bounded on $[0,T]$. Furthermore there is a $N_0\in\mathbb{N}$ depending only on $C$ and $T$  such that for all $N\geq N_0$
	\begin{equation}\label{priorres}
	\sup\limits_{t\in[0,T]}e(t)\leq  N^{-\delta}.
	\end{equation}
\end{lem}
\begin{proof}
	This lemma has been proved in \cite[Lemma 3.3]{vlasovhui}. For completeness, we provide a proof here, which is done by contradiction. We assume that there is a $t\in [0,T]$ with $e(t)\geq N^{-\lambda_2}$ and show that for $N\geq N_0$ with some $N_0\in\mathbb{N}$ specified below, we get a contradiction. 
	
	It follows that the infimum over all times $t$ where $e(t)$ is larger than or equal to  $N^{-\lambda_2}$ exists and we define  $$T_*=\inf_{}\{0\leq t \leq T:e(t)\geq N^{-\lambda_2}\}.$$
	We get by  continuity of $e(t)$ together with $e(0)=0$ that $T^*>0$,
	\begin{equation}\label{contra}e(T_*)=N^{-\lambda_2}\text{ and }\max_{0\leq t \leq T_*}e(t)=N^{-\lambda_2}\;.
	\end{equation}
	
	Since  \eqref{priorcon} implies \eqref{priorres}, we get for $T_1=T_*$ that
	
	$$\frac{de(t)}{dt}\leq C\sqrt{\log(N)}e(t)+C\log^2(N)N^{-\lambda_3},\quad 0<t\leq T_*.$$
	Gronwall's Lemma gives that
	$$e(t)\leq e^{C\sqrt{\log(N)}t}\log^2(N)N^{-\lambda_3},$$
	in particular $$e(T_*)\leq e^{C\sqrt{\log(N)}T_*}\log^2(N)N^{-\lambda_3}.$$
	
	Since $ e^{C\sqrt{\log(N)}T_*}$ and $\log^2(N)$ are  asymptotically bounded by any positive power of $N$, we can find a $N_0\in\mathbb{N}$  depending only on $C$ and $T_*$ such that for any $N\geq N_0$
	$$ e^{C\sqrt{\log(N)}T_*}\log^2(N)< N^{\lambda_3-\lambda_2},\quad \mbox{ for }0<\lambda_2<\lambda_3,$$
	and hence
	$$e(T_*)< N^{-\lambda_2}\text{ for any }N\geq N_0\:.$$
	
	Thus  we get a contradiction to \eqref{contra} for all $N\geq N_0$ and the lemma is proven.
\end{proof}

Our next theorem states that the $N$-particle trajectory $X_t=(X_i^t)_{i=1,\cdots,N}$ starting from $X_0$ (i.i.d. with common density $\rho_0$) remains close to the mean-field trajectory $Y_t=(Y_i^t)_{i=1,\cdots,N}$ with the same initial configuration $Y_0=X_0$. More precisely, we prove that the measure of the set where the maximal distance $\sup\limits_{t\in[0,T]}\norm{X_t-Y_t}_\infty$ on $[0,T]$ excedes $N^{-\delta}$ decreases exponentially with the number of particles $N$, as $N$ grows to infinity.
\begin{thm}\label{thmmean}(Convergence)
	Assume that trajectories $X_t=(X_i^t)_{i=1,\cdots,N}$, $Y_t=(Y_i^t)_{i=1,\cdots,N}$ satisfy \eqref{particl system} and \eqref{meandynamics} respectively with the initial data $X_0=Y_0$, which is i.i.d. sharing the common density $\rho_0$ satisfying \eqref{initial}.  Then for any $\alpha>0$, there exist some constant $N_0>0$ depending only on $\nu$, $\alpha$, $T$ and $C_{\rho_0}$, such that for $N\geq N_0$, the following estimate holds with the cut-off index $0<\delta<\frac{1}{d}$
	$$\mathbb{P}\left(\sup\limits_{t\in[0,T]}\norm{X_t-Y_t}_\infty \leq N^{-\delta} \right)\geq 1-N^{-\alpha}.$$
\end{thm}
\begin{proof}
	We can prove the convergence result by using the consistency from Proposition \ref{propconsis}, the stability from Proposition \ref{propstab} and Lemma \ref{lmgron}.
Denote the event
\begin{equation}
\mathcal{C} :=\left\{\sup\limits_{t\in[0,T]}\nel \mathcal{F}^N(Y_t)-\overline {\mathcal{F}}^N(Y_t)\ner_\infty\leq C_{2,\alpha}  \nu^{\frac{1}{2}}N^{\frac{\delta(d-2)-1}{2}}\log (N)\right\}.
\end{equation}
Consider the quantity $e(t)$ defined as
\begin{equation}
e(t):=\nel X_t-Y_t\ner_\infty.
\end{equation}
Computing under the event $\mathcal{C}\cap \mathcal{S}(C_{3,\alpha})$ and using the fact $\frac{d\|x\|_{\infty}}{dt}\leq \|\frac{dx}{dt}\|_{\infty}$, one has
\begin{align}\label{et}
\frac{d e(t)}{dt}&\leq\nel \mathcal{F}^N(X_t)-\overline {\mathcal{F}}^N(Y_t)\ner_\infty \notag \\
&\leq \nel \mathcal{F}^N(X_t)- {\mathcal{F}}^N(Y_t)\ner_\infty +\nel \mathcal{F}^N(Y_t)-\overline {\mathcal{F}}^N(Y_t)\ner_\infty \notag\\
&\leq C_{3,\alpha}\log(N)\nel X_t-Y_t\ner_\infty+C_{3,\alpha} \nu^{\frac{1}{2}}\log(N)N^{-\frac{\beta+2}{2d}}+C_{2,\alpha}  \nu^{\frac{1}{2}}N^{\frac{\delta(d-2)-1}{2}}\log (N)\notag \\
&\leq C(\alpha, T, C_{\rho_0})\log(N)e(t)+C(\alpha, T, C_{\rho_0}) \nu^{\frac{1}{2}}N^{\frac{\delta(d-2)-1}{2}}\log (N).
\end{align}
According to Proposition \ref{propstab} one has
\begin{equation}\label{93}
\mathcal{C}\cap\bigcap\limits_{n=0}^{M-1}\B_{t_n}\cap\mathcal{H}\cap \A \subset\mathcal{C} \cap \mathcal{S}(C_{3,\alpha}).
\end{equation}

Thus it follows from \eqref{et} that for any $0<T_1\leq T$, it holds
\begin{equation}
\frac{d e(t)}{dt}\leq C\log (N)e(t)+C(\alpha, T, C_{\rho_0}) \nu^{\frac{1}{2}}N^{-\lambda}\log (N),\quad \mbox{for all } t\in(0,T_1],
\end{equation}
under the event $\mathcal{C}\cap\bigcap\limits_{n=0}^{M-1}\B_{t_n}\cap\mathcal{H}\cap \A$, where 
\begin{equation}
-\lambda:=\frac{\delta(d-2)-1}{2}.
\end{equation}
And for $0<\delta<\frac{1}{3}$ we have $-\lambda<-\delta$.

Recall the event
\begin{equation}
\mathcal{A}:=\left\{\sup\limits_{t\in[0,T]}e(t)\leq N^{-\delta}\right\}\subseteq \left\{\sup\limits_{t\in[0,T_1]}e(t)\leq N^{-\delta}, \mbox{ for any }0<T_1\leq T\right\}.
\end{equation}
We deliberately  take the event $\A$ out as the condition \eqref{priorcon} in Lemma \ref{lmgron}. Hence it yields that
\begin{equation}
\sup\limits_{t\in[0,T]} e(t)\leq N^{-\delta} 
\end{equation}
under the event $\mathcal{C}\cap\bigcap\limits_{n=0}^{M-1}\B_{t_n}\cap\mathcal{H}$. Then we arrive at that
\begin{align*}
\PP\left(\sup\limits_{t\in[0,T]} \nel X_t-Y_t\ner_\infty\geq N^{-\delta} \right)&\leq  \sum\limits_{n=0}^{M-1}\PP(\B_{t_n}^c)+\PP(\mathcal{H}^c)+\PP(\mC^c)\notag\\
&\leq  TN^{\frac{\beta}{d}-\alpha}+C_BN^{\frac{2+\beta}{2d}}\exp(-C_BN^{\frac{\beta-2}{d}})+N^{-\alpha}\leq N^{-\alpha'},
\end{align*}
by using Proposition \ref{propconsis}, Lemma \ref{lmB} and Lemma \ref{lmlarge}. Since $\alpha>0$ is arbitrary and so is $\alpha'$, we have proved Theorem \ref{thmmean}.
\end{proof}

\subsection{The error estimate on interaction}
Using Theorem \ref{thmmean}, we obtain the error estimate on interaction:
\begin{thm}\label{thminteract}
	Under the same assumption as Theorem \ref{thmmean}, let $\rho(x,t)$ be the regular solution to the aggregation equation \eqref{RKS} up to time $T$ such that $\rho\in L^\infty(0,T;L^1\cap L^\infty(\RR^d))$. Assume that $\{X_i^{t}\}_{i=1}^N$ satisfy the particle system \eqref{particl system} and $F^N$ satisfies \eqref{kernel}. Then for any $\alpha>0$, there exists some constants $C_{4,\alpha}>0$ depending only on $\alpha$, $T$ and $C_{\rho_0}$ such that the following estimate holds with the cut-off index $0<\delta<\frac{1}{3}$
	\begin{align*}
\PP\bigg(&\sup\limits_{t\in[0,T]}\sup\limits_{i=1,\cdots,N}\left|\int_{\RR^d}F^N(X_i^t-y)\rho(y,t)dy-\frac{1}{N-1}\sum_{j\neq i}^{N}F^N(X_i^t-X_j^{t})\right|\notag\\
&\quad \quad \leq C_{4,\alpha}  \nu^{\frac{1}{2}}N^{-\delta}\log(N)\bigg) \geq 1-N^{-\alpha},
\end{align*}
\end{thm}
\begin{proof}
	For $i=1$, let us denote
	\begin{align*}
   e_1^t:=\left|\int_{\RR^d}F^N(X_1^t-y)\rho(y,t)dy-\frac{1}{N-1}\sum_{j=2}^{N}F^N(X_1^t-X_j^{t})\right|,
	\end{align*}
	then one splits it into two parts:
\begin{align*}
e_1^t&\leq \left|\int_{\RR^d}F^N(X_1^t-y)\rho(y,t)dy-\frac{1}{N-1}\sum_{j=2}^{N}F^N(X_1^t-Y_j^{t})\right|\notag \\
&\quad+ \left|\frac{1}{N-1}\sum_{j=2}^{N}F^N(X_1^t-Y_j^{t})-\frac{1}{N-1}\sum_{j=2}^{N}F^N(X_1^t-X_j^{t})\right|\notag\\
&=: e_{11}^t+e_{12}^t,
\end{align*}
where $Y_t=(Y_i^t)_{i=1,\cdots,N}$ satisfies  \eqref{meandynamics}.

To estimate $e_{11}^t$, we use the law of large number estimates. In particular, similar to the estimate \eqref{largef} in Lemma \ref{lmlarge}, we can prove that  at any fix time $t\in[0,T]$
\begin{equation}\label{fixtime}
\PP\left(\left|\int_{\RR^d}F^N(X_1^t-y)\rho(y,t)dy-\frac{1}{N-1}\sum_{j=2}^{N}F^N(X_1^t-Y_j^{t})\right|\geq CN^{\frac{\delta(d-2)-1}{2}}\log (N)\right)\leq N^{-\alpha},
\end{equation}
where $C$ depends only on $\alpha$, $T$ and $C_{\rho_0}$.  Then following the procedure in Proposition \ref{propconsis}, we can get the estimate of $e_{11}^t$ for all the time $t\in[0,T]$. Hence one has
\begin{equation}\label{e11estimate}
\PP\left(\sup\limits_{t\in[0,T]} e_{11}^t\geq C \nu^{\frac{1}{2}}N^{\frac{\delta(d-2)-1}{2}}\log (N)\right)\leq N^{-\alpha},
\end{equation}
where $C$ depends only on  $\alpha$, $T$ and $C_{\rho_0}$. 

To estimate $e_{12}^t$, we shall use the result from Theorem \ref{thmmean}.  Let us recall the event
\begin{equation}
\mathcal{A}=\left\{\sup\limits_{t\in[0,T]}\nel X_t-Y_t\ner_\infty\leq  N^{-\delta}\right\},
\end{equation}
then it follows from Theorem \ref{thmmean} that
\begin{equation}
\PP(\mathcal{A}^c)\leq N^{-\alpha}.
\end{equation}
For any $\xi\in\RR^d$ with $|\xi|<4N^{-\delta}$, it follows from \cite[Lemma 6.3]{lazarovici2015mean} that
\begin{equation}
|F^N(x+\xi)-F^N(x)|\leq CL^N(x)|\xi|,
\end{equation}
where $L^N$ is defined in \eqref{lN}.  Therefore, it holds
\begin{align}\label{e12}
&\frac{1}{N-1}\sum_{j=2}^{N}\left|F^N(X_1^t-Y_j^{t})-F^N(X_1^t-X_j^{t})\right|\notag\\
\leq & \frac{1}{N-1}\sum_{j=2}^{N} CL^N(X_1^t-Y_j^{t})|X_j^{t}-Y_j^{t}|\leq C\frac{1}{N-1}\sum_{j=2}^{N} L^N(X_1^t-Y_j^{t})\nel X_t-Y_t\ner_\infty\notag\\
\leq &CN^{-\delta}\frac{1}{N-1}\sum_{j=2}^{N} L^N(X_1^t-Y_j^{t}),
\end{align}
under the event $\A$. 
Next we denote the event
\begin{equation*}
\B_1:=\left\{\sup\limits_{t\in[0,T]}\left|\frac{1}{N-1}\sum_{j=2}^{N} L^N(X_1^t-Y_j^{t})-\int_{\RR^d}L^N(X_1^t-y)\rho(y,t)dy\right|\leq C  \nu^{\frac{1}{2}} N^{\frac{d\delta-1}{2}}\log (N)\right\},
\end{equation*}
Similar to the law of large numbers estimate \eqref{largel} in Lemma \ref{lmlarge}, we can prove that
\begin{equation}
\PP\left(\B_1^c\right)\leq  N^{-\alpha}.
\end{equation}
 Hence it follows from \eqref{e12} and Lemma \ref{converse} that
\begin{align*}
\sup\limits_{t\in[0,T]}e_{12}^t&\leq C\left(\left|\int_{\RR^d}L^N(X_1^t-y)\rho(y,t)dy\right|+C \nu^{\frac{1}{2}}N^{\frac{d\delta-1}{2}}\log (N)\right) N^{-\delta}\notag\\
&\leq C \nu^{\frac{1}{2}}\log(N)N^{-\delta},
\end{align*}
under the event $\A\cap \B_1$,
which implies that
\begin{equation}\label{e12estimate}
\PP\left(\sup\limits_{t\in[0,T]}e_{12}^t\leq C \nu^{\frac{1}{2}}\log(N)N^{-\delta} \right)\geq 1-N^{-\alpha},
\end{equation}
where $C$ depends only on $\nu$, $\alpha$, $T$ and $C_{\rho_0}$.
 
 Collecting estimates \eqref{e11estimate} and \eqref{e12estimate}, it yields that
\begin{equation}\label{e2}
\PP\left(\sup\limits_{t\in[0,T]}e_{1}^t\leq C \nu^{\frac{1}{2}}N^{-\delta}\log(N) \right)\geq 1-N^{-\alpha}.
\end{equation}
where $C$ depends only on $\alpha$, $T$ and $C_{\rho_0}$. Similarly, we can arrive at the same estimate for $i=2,\cdots,N$, which  finishes the proof.
\end{proof}

\section{Parameter estimation and the proof Theorem \ref{mainthm}}
In this section, we obtain the diffusion parameter estimation and prove our main Theorem \ref{mainthm}.

Let us recall \eqref{split} that
\begin{equation}
|\hat{\nu}-\nu|\leq  C \nu^{\frac{1}{2}}(|\mathcal{I}_2|^{\frac{1}{2}}+|\mathcal{I}_3|^{\frac{1}{2}})+|\nu_{K,N}-\nu|\,,
\end{equation}
where
\begin{equation}\label{es1}
\nu_{K,N}:=\frac{1}{2dKT} \sum_{i=1}^K\sum_{n=0}^{M-1} \left| X_{i}^{(n+1)}-X_{i}^{(n)}-\int_{t_n}^{t_{n+1}}\frac{1}{N-1}\sum_{j\neq i}^{N}F^N\big(X_{i}^s-X_{j}^s\big)\,ds\right|^2\,,
\end{equation}
and
\begin{equation}
|\mathcal{I}_2|=\frac{1}{dKT} \sum_{i=1}^K\sum_{n=0}^{M-1} \left| \int_{t_n}^{t_{n+1}}\left(\frac{1}{N-1}\sum_{j\neq i}^{N}F^N\big(X_{i}^s-X_{j}^s\big) -\int_{\mathbb{R}^d}F^N(X_i^{s}-y)\rho(y,s)dy\right)ds\right|^2,
\end{equation}
and
\begin{equation}
|\mathcal{I}_3|=\frac{1}{dKT} \sum_{i=1}^K\sum_{n=0}^{M-1}\left|\int_{t_n}^{t_{n+1}}\int_{\mathbb{R}^d}F^N(X_i^{s}-y)\rho(y,s)\,dyds\right|^2.
\end{equation}
According to Lemma \ref{converse}, one has
\begin{equation}\label{I3}
|\mathcal{I}_3|\leq C\Delta t,
\end{equation}
where $C$ depends only on $T$ and $C_{\rho_0}$. Then it follows from Theorem \ref{thminteract} that
\begin{equation}\label{I2}
\PP\left(|\mathcal{I}_2|\leq C\nu\Delta t N^{-2\delta}\log^2(N)\right)\geq 1-N^{-\alpha},
\end{equation}
where $C$ depends only on $\alpha$, $T$ and $C_{\rho_0}$.
It is left to estimate the error between $\nu_{K,N}$ and $\nu$, which can be done by using the concentration property of $\chi^2$ random variable.
\begin{thm}\label{thm3}
	Under the assumption as in Theorem \ref{mainthm}. Suppose that $\nu_{K,N}$ satisfies \eqref{es1}, then the following estimate holds
	\begin{equation}\label{thm3eq}
	\PP\left(|\nu_{K,N}-\nu|>\gamma \nu\right)\leq 2e^{-\frac{dKM\gamma^2}{8}},\quad \text{for any }\gamma\in(0,1).
	\end{equation}
\end{thm}

\begin{proof}
	Recall that
	\begin{equation}
	X_{i}^{(n+1)}=X_{i}^{(n)}+\int_{t_n}^{t_{n+1}}\frac{1}{N-1}\sum_{j\neq i}^{N}F^N\big(X_{i}^s-X_{j}^s\big)ds+\sqrt{2\nu\Delta t}\,\mathcal{N}_i^{(n)},\quad i=1,\cdots,K,
	\end{equation}
	then we know
	\begin{equation}
	\frac{X_{i}^{(n+1)}-X_{i}^{(n)}-\int_{t_n}^{t_{n+1}}\frac{1}{N-1}\sum_{j\neq i}^{N}F^N\big(X_{i}^s-X_{j}^s\big)ds}{\sqrt{2\nu\Delta t}} \sim \mathcal{N}(0, 1)^d.
	\end{equation}
	
	Notice that the random variable $$S:=\frac{1}{2\nu \Delta t} \sum_{i=1}^K\sum_{n=0}^{M-1} \left| X_{i}^{(n+1)}-X_{i}^{(n)}-\int_{t_n}^{t_{n+1}}\frac{1}{N-1}\sum_{j\neq i}^{N}F^N\big(X_{i}^s-X_{j}^s\big)ds\right|^2$$ is distributed according to the chi-squared distribution with $dNM$ degrees of freedom. This is usually denoted as
	\begin{equation}
	S~\sim~\chi^2(dKM).
	\end{equation}
	Recall a simple fact from probability theory, we know $\EE[S]=dKM$ and
	\begin{equation}
	\mbox{Var}[S]=\EE\left[(S-dKM)^2\right]=2dKM.
	\end{equation}
	Recall that the estimate of $\nu$ is given by
	\begin{equation}
	\nu_{K,N}=\frac{1}{2dKT} \sum_{i=1}^K\sum_{n=0}^{M-1} \left| X_{i}^{(n+1)}-X_{i}^{(n)}-\int_{t_n}^{t_{n+1}}\frac{1}{N-1}\sum_{j\neq i}^{N}F^N\big(X_{i}^s-X_{j}^s\big)ds\right|^2,
	\end{equation}
	which leads to
	\begin{equation}
	\EE\left[\left(\frac{\nu_{K,N}}{\nu}-1\right)^2\right]=\frac{2}{dKM}.
	\end{equation}
	Hence we have
	\begin{equation}
	\EE\left[(\nu_{K,N}-\nu)^2\right]=\frac{2\nu^2}{dKM}.
	\end{equation}
	
	Also by the concentration of $\chi^2$ variable, we have the following two sided tail bound
	\begin{equation}
	\PP\left(\left|\frac{S}{dKM}-1\right|>\gamma\right)\leq 2e^{-\frac{dKM\gamma^2}{8}},\quad \text{for any }\gamma\in(0,1),
	\end{equation}
	which is a direct result from the Bernstein's inequality as the form showed in \cite[Corollary 2.11]{BSLGM}. And it leads to
	\begin{equation}\label{I1}
	\PP\left(|\nu_{K,N}-\nu|>\gamma \nu\right)\leq 2e^{-\frac{dKM\gamma^2}{8}},\quad \text{for any }\gamma\in(0,1).
	\end{equation}
	Hence it concludes the proof.
\end{proof}

Collecting estimates \eqref{I1}, \eqref{I2} and \eqref{I3}, one has
\begin{equation}
\PP\left(|\wh{\nu}-\nu|\leq C\nu^{\frac{1}{2}}\Delta t^\frac{1}{2}  (1+ \nu^{\frac{1}{2}}N^{-\delta}\log(N))+\gamma \nu\right)\geq 1-N^{-\alpha}-2e^{-\frac{dKM\gamma^2}{8}},
\end{equation}
for any $\gamma\in(0,1)$. Hence Theorem \ref{mainthm} has been proved.

\section{Extension to regular interacting kernel $F\in W^{1,\infty}(\RR^d)$}
In this section, we will extend our result to the particle system with regular interacting force $F$, which satisfies
\begin{equation}\label{regularforce}
F\in W^{1,\infty} (\mathbb{R}^d).
\end{equation}
Since $F$ is non-singular, there is no need to mollify the force $F$ anymore. To be more specific, we consider trajectories $\{X_i^t\}_{i=1}^N$ satisfying SDEs:
\begin{equation}\label{random particle}
dX_{i}^t=\frac{1}{N-1}\sum_{j\neq i}^{N}F\big(X_{i}^t-X_{j}^t\big)\,dt+\sqrt{2\nu}\,dB_i^t,\quad i=1,\cdots,N,
\end{equation}
where the initial data $\{X_i^0\}_{i=1}^{N}$ are i.i.d. sharing the common density $\rho_0\in L^1\cap L^\infty(\mathbb{R}^d)$.  Then the solution $\rho$ to the mean field equation:
\begin{subequations}\label{KS2}
	\begin{align}
	&\partial_t\rho=\nu\Delta\rho-\nabla\cdot(\rho F\ast \rho),\quad x\in\mathbb{R}^d,~t>0,\\
	&\rho(x,0)=\rho_0(x), 
	\end{align}
\end{subequations}
has the following regularity for any $T>0$
\begin{equation}
\|\rho\|_{L^\infty\left(0,T; L^1\cap L^\infty (\mathbb{R}^d)\right)}\leq C\left(T,\|\rho_0\|_{L^1\cap L^\infty(\mathbb{R}^d)},\norm{F}_{W^{1,\infty}(\mathbb{R}^d)}\right)=:C_{F,\rho_0}.
\end{equation}

Take a time step $\Delta t>0$ and let $t_n:=n\Delta t$
and $M:=\frac{T}{\Delta t}$ (we assume that $\frac{T}{\Delta t}$ is an
integer). Denote $X_i^{(n)}:=X_i^{t_n}=X_i^{n\Delta t}$ as the solution to \eqref{random particle} at time $t_n$. Namely, one has
\begin{align}
X_i^{(n+1)}-X_i^{(n)}= \int_{t_n}^{t_{n+1}}\frac{1}{N-1}\sum_{j\neq i}^{N}F\big(X_{i}^s-X_{j}^s\big)\,ds+\sqrt{2\nu \Delta t}\mathcal{N}_i^{(n)},
\end{align}
where $\mathcal{N}_i^{(n)}\sim \mathcal{N}(0,1)^d$, i.e.  the standard Gaussian distribution in dimension $d$.

Then 
we are ready to define our estimator for the diffusion parameter as before
\begin{equation}\label{estimator1'}
\widehat\nu:=\frac{1}{2dKT} \sum_{i=1}^K\sum_{n=0}^{M-1} \left | X_{i}^{(n+1)}-X_{i}^{(n)}\right|^2,
\end{equation}
where $1\ll K \ll N$, which means we only have partial observations.

The extended result can be described in the following theorem.
\begin{thm}\label{mainthm'}
	Suppose that $F(x)\in W^{1,\infty}(\RR^d)$ and $0\leq\rho_0(x)\in L^1\cap L^\infty(\mathbb{R}^d)$. For any $T>0$, take a time step $\Delta t>0$ and define $t_n:=n\Delta t$ 
	and $M:=\frac{T}{\Delta t}$. Let $\{X_i^{(n)}\}_{i=1,n=0}^{K,M}$ be the sample trajectories satisfying \eqref{random particle} at time $t_n$. Then  there exists some constant $N_0>0$  depending only on $\nu$, $\alpha$, $T$, $\norm{F}_{W^{1,\infty}(\RR^d)}$ and $\|\rho_0\|_{L^1\cap L^\infty(\mathbb{R}^d)}$, such that for $N\geq N_0$,  the estimator $\widehat{\nu}$ defined in \eqref{estimator1'} is an approximation of $\nu$, and the following estimate holds
	\begin{align}\label{maineq'}
	\PP\left(|\widehat{\nu}-\nu|\leq C_\alpha \nu^{\frac{1}{2}}\Delta t^{\frac{1}{2}} (1+\nu^{\frac{1}{2}}N^{-\frac{1}{2}}\log(N))+\nu \gamma\right)
\geq1-N^{-\alpha}-2e^{-\frac{dKM\gamma^2}{8}}, 
	\end{align}
	for any $\gamma\in(0,1)$, where $C_\alpha>0$  depends only on $\alpha$, $T$, $\norm{F}_{W^{1,\infty}(\RR^d)}$ and $\|\rho_0\|_{L^1\cap L^\infty(\mathbb{R}^d)}$. 
	In particular, let $N$ goes to infinity and choose $\Delta t^{\frac{1}{2}}=\gamma$, it follows from \eqref{maineq'} that
\begin{align}\label{mianeq1'}
	&\PP\left(|\widehat{\nu}-\nu|\leq C_\alpha(\nu^{\frac{1}{2}}+\nu) \Delta t^{\frac{1}{2}}\right)
	\geq1-2e^{-\frac{dKT}{8}}.
	\end{align}
\end{thm}
\begin{proof} 
	
	Again, we defined a intermediate estimator
	\begin{equation}
\nu_{K,N}:= \frac{1}{2dKT} \sum_{i=1}^K\sum_{n=0}^{M-1} \left \lvert X_{i}^{(n+1)}-X_{i}^{(n)}- \int_{t_n}^{t_{n+1}}\frac{1}{N-1}\sum_{j\neq i}^{N}F\big(X_{i}^s-X_{j}^s\big)\,ds\right \rvert^2 
	\end{equation}
	then we split the error into two parts:
	\begin{align}\label{split1'}
	|\hat{\nu}-\nu|\leq |\hat{\nu}-\nu_{K,N}|+|\nu_{K,N}-\nu|.
	\end{align}
	and we can prove that there exists a positive number $C$ such that
	\begin{equation}\label{split'}
	|\widehat{\nu}-\nu|\leq C \nu^{\frac{1}{2}}(|\mathcal{I}_2|^{\frac{1}{2}}+|\mathcal{I}_3|^{\frac{1}{2}})+|\nu_{K,N}-\nu|.
	\end{equation}
	with
	\begin{equation}
	|\mathcal{I}_2|:=\frac{1}{dKT}\sum_{i=1}^K\sum_{n=0}^{M-1} \left| \int_{t_n}^{t_{n+1}}\left(\frac{1}{N-1}\sum_{j\neq i}^{N}F\big(X_{i}^s-X_{j}^s\big) -\int_{\mathbb{R}^d}F(X_i^{s}-y)\rho(y,s)dy\right)ds\right|^2 ,
	\end{equation}
	and 
	\begin{equation}
	|\mathcal{I}_3|:=\frac{1}{dKT} \sum_{i=1}^K\sum_{n=0}^{M-1}\left|\int_{t_n}^{t_{n+1}}\int_{\mathbb{R}^d}F(X_i^{s}-y)\rho(y,s)dyds\right|^2.
	\end{equation}

	According to Lemma \ref{converse}, one has
	\begin{equation}\label{I3'}
	|\mathcal{I}_3|\leq C\Delta t,
	\end{equation}
	where $C$ depends only on $T$, $\norm{F}_{W^{1,\infty}(\RR^d)}$ and $\|\rho_0\|_{L^1\cap L^\infty(\mathbb{R}^d)}$. It follows from Theorem \ref{thm3} that
	\begin{equation}\label{I1'}
\PP\left(|\nu_{K,N}-\nu|>\gamma \nu\right)\leq 2e^{-\frac{dKM\gamma^2}{8}},\quad \text{for any }\gamma\in(0,1).
\end{equation}
	
	Now it is left to get the estimate of $\mathcal{I}_2$. The main idea behind the proof is also to construct a mean-field dynamic system $\big\{Y_i^{t}\big\}_{i=1}^N$ without interaction:
	\begin{equation}\label{meandynamics'}
	dY_{i}^t=\int_{\RR^d}F\big(Y_i^t-y\big)\rho(y,t)dy\,dt+\sqrt{2\nu}\,dB_i^t,\quad i=1,\cdots,N,
	\end{equation}
	here again we let $\{Y_i^t\}_{i=1}^N$ has the same initial condition as $\{X_i^t\}_{i=1}^N$ (i.i.d. with common density $\rho_0$) .
	Consider the quantity $e(t)$ defined as
	\begin{equation}
	e(t):=\nel X_t-Y_t\ner_\infty.
	\end{equation}
	Following the same procedure as  in Lemma \ref{lmlarge} and Proposition \ref{propconsis}, one can prove  that there exists some $C_{1,\alpha}$ depending only on $\alpha$, $T$, $\norm{F}_{W^{1,\infty}(\RR^d)}$ and $\|\rho_0\|_{L^1\cap L^\infty(\mathbb{R}^d)}$ such that
			\begin{equation}
	\PP\left(\sup\limits_{t\in[0,T]}\nel \mathcal{F}(Y_t)-\overline{ \mathcal{F}}(Y_t)\ner_\infty\geq C_{1,\alpha}\nu^{\frac{1}{2}}N^{-\frac{1}{2}}\log (N)\right)\leq N^{-\alpha},
	\end{equation}
	where 
	\begin{equation}
	(\mathcal{F}(Y_t))_i:=\frac{1}{N-1}\sum_{j\neq i}^{N}F\big(Y_{i}^t-Y_{j}^t\big),
	\end{equation}
	and 
	\begin{equation}
	(\overline{ \mathcal{F}}(Y_t))_i:=\int_{\RR^d}F\big(Y_i^t-y\big)\rho(y,t)dy.
	\end{equation}
  We denote the event 
  	\begin{equation}
  \mathcal{C}:=\left\{\sup\limits_{t\in[0,T]}\nel\mathcal{F}(Y_t)-\overline{ \mathcal{F}}(Y_t)\ner_\infty\leq C_{1,\alpha}\nu^{\frac{1}{2}}N^{-\frac{1}{2}}\log (N)\right\},
  \end{equation}
Then using the fact $\frac{d\|x\|_{\infty}}{dt}\leq \|\frac{dx}{dt}\|_{\infty}$, one concludes that under the event $\mathcal{C}$
	\begin{align}\label{et‘}
	\frac{d e(t)}{dt}&\leq\nel \mathcal{F}(X_t)-\overline{ \mathcal{F}}(Y_t)\ner_\infty \notag \\
	&\leq \nel \mathcal{F}(X_t)- \mathcal{F}(Y_t)\ner_\infty +\nel\mathcal{F}(Y_t)-\overline{ \mathcal{F}}(Y_t)\ner_\infty \notag\\
	&\leq C \|X_t-Y_t\|_\infty+C\nu^{\frac{1}{2}}N^{-\frac{1}{2}}\log (N),
	\end{align}
  which leads to
  \begin{equation}\label{XY}
  \sup\limits_{t\in[0,T]} \nel X_t-Y_t\ner_\infty\leq C\nu^{\frac{1}{2}}N^{-\frac{1}{2}}\log (N),
  \end{equation}
  where $C$ depends only on $\alpha$, $T$, $\norm{F}_{W^{1,\infty}(\RR^d)}$ and $\|\rho_0\|_{L^1\cap L^\infty(\mathbb{R}^d)}$. Based on this mean-field limit result, we can prove error estimate on interaction as in Theorem \ref{thminteract}.

  Let us split the error
  \begin{align*}
  &\left|\int_{\RR^d}F(X_1^t-y)\rho(y,t)dy-\frac{1}{N-1}\sum_{j=2}^{N}F(X_1^t-X_j^{t})\right|\notag\\
  \leq& \left|\int_{\RR^d}F(X_1^t-y)\rho(y,t)dy-\frac{1}{N-1}\sum_{j=2}^{N}F(X_1^t-Y_j^{t})\right|\notag \\
  &+ \left|\frac{1}{N-1}\sum_{j=2}^{N}F(X_1^t-Y_j^{t})-\frac{1}{N-1}\sum_{j=2}^{N}F(X_1^t-X_j^{t})\right|\notag\\
  =:&e_{11}^t+e_{12}^t.
  \end{align*}
  Similar to estimates \eqref{e11estimate} and \eqref{e12estimate},  it is easy to compute that
  \begin{equation}\label{part1}
\PP\left(\sup\limits_{t\in[0,T]}e_{11}^t\leq C\nu^{\frac{1}{2}}N^{-\frac{1}{2}}\log (N) \right)\geq 1-N^{-\alpha},
  \end{equation}
  and
  \begin{equation}\label{part2}
  \PP\left(\sup\limits_{t\in[0,T]}e_{12}^t\leq C \nu^{\frac{1}{2}}N^{-\frac{1}{2}}\log(N) \right)\geq 1-N^{-\alpha}.
  \end{equation}
  where $C$ depends only on $\alpha$, $T$, $\norm{F}_{W^{1,\infty}(\RR^d)}$ and $\|\rho_0\|_{L^1\cap L^\infty(\mathbb{R}^d)}$.
  
  Combining \eqref{part1} and \eqref{part2}, it leads to
  \begin{align*}
  \PP\bigg(\sup\limits_{t\in[0,T]}&\left|\int_{\RR^d}F(X_1^t-y)\rho(y,t)dy-\frac{1}{N-1}\sum_{j=2}^{N}F(X_1^t-X_j^{t})\right|\notag\\
  &\leq C N^{-\frac{1}{2}}\log(N) \bigg)\leq 1-N^{-\alpha},
  \end{align*}
  which yields
  \begin{equation}\label{I2'}
  \PP\left(|\mathcal{I}_2|\leq C\nu^{\frac{1}{2}}\Delta tN^{-1}\log^2(N) \right)\geq 1- N^{-\alpha},
  \end{equation}
  where $C$ depends only on $\alpha$, $T$, $\norm{F}_{W^{1,\infty}(\RR^d)}$ and $\|\rho_0\|_{L^1\cap L^\infty(\mathbb{R}^d)}$.
  
  Collecting \eqref{I1'}, \eqref{I2'} and \eqref{I3'}, we obtain our result
  \begin{align*}
  	\PP\left(|\widehat{\nu}-\nu|\leq C \nu^{\frac{1}{2}} \Delta t^{\frac{1}{2}}(1+\nu^{\frac{1}{2}}N^{-\frac{1}{2}}\log(N))+\nu \gamma\right)
  \geq1-N^{-\alpha}-2e^{-\frac{dKM\gamma^2}{8}},
  \end{align*}
for any $\gamma\in(0,1)$.
\end{proof}

{\bf Acknowledgments:}
We would like to thank Zhenfu Wang for his careful proofreading of the draft and useful suggestions. H.H. is partially supported by NSFC (Grant No. 11771237). The research of J.-G. L. is partially supported by  KI-Net NSF RNMS (Grant No. 1107444) and NSF DMS
(Grant No. 1812573). The work of J.L. is supported in part by the NSF DMS (Grant NO. 1454939).

\bibliography{keller}
\bibliographystyle{abbrv}
\end{document}